\documentclass[12 pt]{amsart}  
\usepackage{amsthm, amsmath, amscd, amssymb, latexsym, stmaryrd, color}

\usepackage[all]{xypic}
\usepackage{color}
\usepackage{fullpage}
\definecolor{hot}{RGB}{65,105,225}
\usepackage[pagebackref=true,colorlinks=true, linkcolor=hot ,  citecolor=hot, urlcolor=hot]{hyperref}

\theoremstyle{plain}

\newtheorem{theorem}{Theorem}[section]
\newtheorem{lemma}[theorem]{Lemma}
\newtheorem{proposition}[theorem]{Proposition}
\newtheorem{prop}[theorem]{Proposition}
\newtheorem{prop-def}[theorem]{Proposition-Definition}

\newtheorem{conjecture}[theorem]{Conjecture}

\theoremstyle{definition}
\newtheorem{rem}[theorem]{Remark}

\newtheorem{example}[theorem]{Example}

\theoremstyle{remark}

\usepackage[mathscr]{eucal}
\usepackage{graphics, graphpap}
\usepackage{array, tabularx, longtable}
\usepackage{color}
\numberwithin{equation}{section}

\def\pr{\mathrm{pr}}

\def\ord{\mathrm{ord}}

\def\Spec{\mathrm{Spec}\,}

\def\Hom{\mathrm{Hom}}

\def\bm{\mathbf{m}}
\def\bk{\mathbf{k}}

\def\bZ{\mathbb{Z}}

\def\cO{\mathcal{O}}
\def\cL{\mathscr{L}}

\def\bC{\mathbb{C}}
\def\bQ{\mathbb{Q}}
\def\bN{\mathbb{N}}
\def\bP{\mathbb{P}}
\def\cM{\mathcal{M}}

\def\ra{\rightarrow}
\def\sX{\mathscr{X}}
\def\sY{\mathscr{Y}}

\def\ra{\rightarrow}
\def\ol{\overline}
\def\Ct{\bC\llbracket t \rrbracket}
\def\bA{\mathbb{A}}
\def\Cont{\text{\rm Cont}}
\def\xra{\xrightarrow}

\title[Cohomology of contact loci]{Cohomology of contact loci}  

\author{Nero Budur}
\address{Nero Budur:
(1) KU Leuven, Celestijnenlaan 200B, B-3001 Leuven, Belgium; (2) BCAM,  Basque Center for Applied Mathematics, Mazarredo 14, 48009 Bilbao, 
Basque Country, Spain.} 
\email{nero.budur@kuleuven.be}

\author{J. Fern\'andez de Bobadilla}
\address{Javier Fern\'andez de Bobadilla:  
(1) IKERBASQUE, Basque Foundation for Science, Maria Diaz de Haro 3, 48013, 
    Bilbao, Basque Country, Spain;
(2) BCAM,  Basque Center for Applied Mathematics, Mazarredo 14, 48009 Bilbao, 
Basque Country, Spain; 
(3) Academic Colaborator at UPV/EHU.} 
\email{jbobadilla@bcamath.org}

\author{Quy Thuong L\^e}
\address{Quy Thuong L\^e: VNU University of Science, Vietnam National University, Hanoi, 
334 Nguyen Trai Street, Thanh Xuan District, Hanoi, Vietnam}
\email{leqthuong@gmail.com}

\author{Hong Duc Nguyen}
\address{Hong Duc Nguyen: (1) BCAM, Basque Center for Applied Mathematics, Alameda de Mazarredo 14, 48009 Bilbao, Bizkaia, Spain; (2) TIMAS, Thang Long University, Nghiem Xuan Yem, Hanoi, Vietnam.} 
\email{hnguyen@bcamath.org $ \textup{\it and} $ nhduc82@gmail.com}



\begin{document}           

\begin{abstract}
We construct a  spectral sequence converging to the cohomology with compact support of the $m$-th contact locus of a complex polynomial.  The first page is explicitly described in terms of a log resolution and coincides with the first page of McLean's spectral sequence converging to the Floer cohomology of the $m$-th iterate of the monodromy, when the polynomial has an isolated singularity. Inspired by this connection, we conjecture that if two germs of holomorphic functions are embedded topologically equivalent, then the Milnor fibers of the their tangent cones are homotopy equivalent.
\end{abstract}

\maketitle                 

\section{Introduction and main results}

The motivation for this note comes from a question of Seidel, and a subsequent question by McLean. In  \cite{DL} Denef and Loeser proved that the Euler characteristic of the contact loci of a complex polynomial $f$ coincides with the Lefschetz number of the $m$-th iterate of the monodromy of $f$. It was observed by Seidel that this in turn coincides with the Euler characteristic of the Floer cohomology of the same monodromy iterate, in the case when $f$ has an isolated singularity, and motivated him to ask what is the relation between the cohomology of the $m$-th contact locus and the  Floer cohomology of the $m$-th iterate of the monodromy.  McLean  \cite{Mclean} constructed a spectral sequence converging to the Floer cohomology of the $m$-th iterate of the monodromy of $f$, whose first page is completely described in terms of a log resolution of $f$, and asked if there is a similar spectral sequence converging to the cohomology of the $m$-th contact locus of $f$. Our main result is an affirmative answer to this question, if one takes compactly supported cohomology for the contact locus. 

It is natural to conjecture that the Floer cohomology of the $m$-th iterate of the monodromy of $f$ is isomorphic to the compactly supported cohomology of the $m$-th contact locus of $f$, and that this isomorphism comes from an isomorphism of McLean spectral sequence with ours. Besides giving  very different interpretations of the same object, if true, our conjecture would endow  Floer cohomology of the $m$-th iterate of the monodromy with a mixed Hodge structure, and would show that McLean's spectral sequence satisfies the same non-trivial degeneration properties as our spectral sequence. 

The conjecture is true if $m$ is the multiplicity of an isolated hypersurface singularity. Inspired by this observation, we conjecture also that if two germs of holomorphic functions are embedded topologically equivalent, then the Milnor fibers of their tangent cones are homotopy equivalent.

To state the main result, let $X$ be a smooth complex algebraic variety of dimension $d$.  For each $m\ge 0$, the $m$-th jet scheme $\cL_m(X)$ of $X$ is the variety parametrizing morphisms $$\gamma:\Spec \mathbb C[t]/(t^{m+1} ) \to X$$ of schemes over $\bC$.  We denote by $\gamma(0)$ the center of a jet $\gamma$, that is, the image in $X$ of the closed point of $\Spec \bC[t]/(t^{m+1})$.

Let $$f\colon X\to \mathbb{C}$$ be a non-invertible regular function, that is, not a unit in the ring of regular functions, or equivalently, a regular function with non-empty zero locus. For a jet $\gamma\in\cL_m(X)$ we denote by $f(\gamma)$ the truncated power series given by the image of $s$ under the morphism of $\bC$-algebras
$$\bC[s]\xrightarrow{\gamma^\#\circ f^\#}\bC[t]/(t^{m+1})$$
 corresponding to the composition $f\circ \gamma$.

Fix a non-empty Zariski closed subset $\Sigma$ in $X_0=f^{-1}(0)$. The {\it $m$-th (restricted) contact locus of $f$} is defined to be 
$$
\mathscr X_{m}(f,\Sigma):=\{\gamma\in \cL_m(X) \mid { \gamma(0)\in \Sigma \text{ and }} f(\gamma)\equiv t^{m}\pmod{t^{m+1}}\}.
$$
For this to be non-empty, we assume that $m>0$.

Let $h:Y\to X$ be a log resolution of $(f,\Sigma)$, that is, a proper morphism from a smooth variety $Y$ such that $E=h^{-1}(X_0)$ and $h^{-1}(\Sigma)$ are divisors with simple normal crossings and the restriction $h \colon Y \setminus E \to X\setminus X_0$ is an isomorphism. By Hironaka, we can and in fact assume that $h$ is a composition of blowing ups of smooth centers. We denote by $E_i$ with $i$ in $S$, the irreducible components of $E$. We define  $$m_i=\ord_fE_i\quad\text{ and }\quad\nu_i=\ord_{K_{Y/X}}E_i +1,$$ 
where $K_{Y /X}$ is the relative canonical divisor defined by the vanishing of $\det dh$. 

We assume that $h$ is {\em $m$-separating}, that is, $m_i+m_j>m$ if $E_i\cap E_j\neq \emptyset$ for all $i\neq j\in S$. By  Lemma \ref{lemMsep} below, such a log resolution always exists. 
We set  
$$A:=\{i\in S\mid h(E_i)\subset\Sigma\},$$
$$S_m:=\{i\in A \mid  m_i \text{ divides } m\},$$ and $$k_i:=m/m_i$$ for each $i\in S_m$ if $S_m$ is non-empty. Fix a tuple of integers $w=(w_i)_{i\in S}$ with $w_i\ge 0$ such that the divisor $$W=-\sum_{i\in S} w_i E_i$$ is relatively very ample for $h$.  We can, and we will assume, that $w_i=0$ if $E_i$ is not an exceptional divisor. For an integer $p$, we let
$$S_{m,p}:=\{i\in S_m \mid  w_ik_i= -p \}$$
if $S_m$ is non-empty, otherwise we set $S_{m,p}$ to be the empty set as well. 

Let $E_i^{\circ}=E_i\setminus \cup_{j\ne i}E_j$. Then there exists an unramified cyclic cover $\tilde{E_i^{\circ}}\ra E_i^\circ$ of degree $m_i$,  given locally in a neighborhood $U$ in $Y$ of a point in $E_i^{\circ}$ by
$$
\{(z,P)\in \mathbb{C}\times (E_i^{\circ}\cap U) \mid z^{m_i}=u(P)^{-1}\},
$$
where $f\circ h=u\cdot y_i^{m_i}$ with $y_i$ a local equation for $E_i$ and $u$ an invertible regular function on $U$. 

Our main result is the following:

\begin{theorem}\label{main}
Let $f:X\to\bC$ be a non-invertible regular function on a smooth variety $X$ of dimension $d$. Let $\Sigma$ be a non-empty closed subset of $f^{-1}(0)$, $m> 0$ an integer, and $h:Y\to X$ an $m$-separating log resolution of $(f,\Sigma)$. 
Then there is a cohomological spectral sequence 
$$E_1^{p,q}=\bigoplus_{i\in S_{m,p}} H_{2(d(m+1)-k_i\nu_i-1)-(p+q)} (\tilde{E^{\circ}_i},\mathbb{Z})\quad\Rightarrow\quad H^{p+q}_c\left(\mathscr X_{m}(f,\Sigma),\mathbb{Z}\right)$$
converging to the cohomology with compact support of the $m$-th contact locus of $f$.
\end{theorem}

\begin{proposition}
\label{degen} After tensoring with $\bQ$, only the first $d$ pages of the spectral sequence $\{E_r, d_r\}_{r\ge 1}$ can contain non-zero differentials.
\end{proposition}

\begin{example}
(i) If $f=x^r\in\bC[x]$ for some integer $r>0$,  and $\Sigma$ is the origin in $X=\bC$, then the identity map is the only log resolution for $(f,\Sigma)$ and it is $m$-separating for any $m>0$. Then
$$
E_1^{p,q}\simeq\left\{
\begin{array}{cl}
 0& \text{ if }r\nmid m\text{ or }(p,q)\neq (0,2(m-\frac{m}{r})),\\ 
 \bZ^r& \text{ if }r\mid m\text{ and }(p,q)=(0, 2(m-\frac{m}{r})).
\end{array}
\right.
$$
Thus the theorem implies that
$$
H_c^{*}\left(\mathscr X_{m}(f,\Sigma),\mathbb{Z}\right)\simeq\left\{
\begin{array}{cl}
 0& \text{ if }r\nmid m\text{ or }*\neq 2(m-\frac{m}{r}),\\ 
 \bZ^r& \text{ if }r\mid m\text{ and } *=2(m-\frac{m}{r}).
\end{array}
\right.
$$
This is compatible with the isomorphism
$$
\mathscr X_{m}(f,\Sigma)\simeq\left\{
\begin{array}{cl}
\emptyset & \text{ if }r\nmid m,\\
\mu_r\times\bC^{m-\frac{m}{r}} & \text{ if }r\mid m,
\end{array}
\right.
$$
where $\mu_r$ is the group of $r$-th roots of unity, which can be easily checked. 

(ii) In the case that $f$ is a product of linear polynomials on $X=\bC^d$, it is shown in \cite{BT} that the spectral sequence degenerates at $E_1$.
\end{example}

 \begin{rem} If $X=\bC^d$ and with  the additional assumption that $f$ has an isolated singularity $\Sigma=\{x\}$, McLean \cite[Theorem 1.2]{Mclean} showed that there exists a spectral sequence 
\begin{equation}\label{eqMcL}
'\!E^{p,q}_1=\bigoplus_{i\in S_{m,p}} H_{d-1-2k_i\nu_i-(p+q)} (\tilde{E^{\circ}_i},\mathbb{Z})\quad \Rightarrow \quad HF^*(\phi^m,+)
\end{equation}
converging to the Floer cohomology of the $m$-th iterate of the monodromy $\phi$ on the Milnor fiber of $f$. We note that $E_1$ in this case differs from $'\!E_1$ by a $(2dm+d-1)$-shift in the total degree $p+q$, hence up to relabelling, the two pages are the same. 
\end{rem}

\begin{conjecture} If $X=\bC^d$ and $f$ has an isolated singularity $\Sigma=\{x\}$, the two spectral sequences $\{E_r, d_r\}_{r\ge 1}$ and $\{'\!E_r, '\!d_r\}_{r\ge 1}$ are isomorphic,  and 
$$
HF^{*}(\phi^m,+) \simeq H^{*+2dm+d-1}_c\left(\mathscr X_{m}(f,\Sigma),\mathbb{Z}\right).
$$
\end{conjecture}

The conjecture is true if $m$ is the multiplicity of $f$ at the singularity. More generally:

\begin{prop}\label{propMult}
Let $$f=f_m+f_{m+1}+\ldots$$ be a polynomial in $d$ variables vanishing at the origin, where $f_i$ are the homogeneous components of degree $i$, and $m>0$ is the multiplicity of $f$ at the origin. Let $\Sigma=\{0\}$. Then,
$$
H_c^*(\mathscr X _m(f,\Sigma),\bZ)\simeq H_{2(dm-1)-*}(F,\bZ),
$$ 
where $F\simeq \{f_m=1\}$ is the Milnor fiber at the origin of the initial form $f_m$ of $f$. If in addition $f$ has an isolated singularity at the origin, then also
$$HF^{*-2dm-d+1}(\phi^m,+,)\simeq H_{2(dm-1)-*}(F,\bZ).$$
\end{prop}

Recall that Zariski's Problem A, the multiplicity conjecture, states that the multiplicity is an embedded topological invariant of a hypersurface singularity. On the other hand, Zariski's Problem B, see \cite{Za}, has counterexamples: there exist embedded topologically equivalent hypersurface singularities (even in families) such that the topology of the tangent cones changes drastically, see \cite{J1,J2}. However, inspired by Proposition \ref{propMult}, we observed that in all these examples the homology of the Milnor fiber of the tangent cone stays the same. Therefore we dare to conjecture that the same is true in general:

\begin{conjecture}
Let $f, g:(\bC^d,0)\ra (\bC,0)$ be two germs of holomorphic functions. If $f$ and $g$ are embedded topologically equivalent, then the Milnor fibers of their initial forms are homotopy equivalent. 
\end{conjecture}

The conjecture holds if $f, g$, and their initial forms have isolated singularities, by \cite[Corollary 2.4]{LR}.

\section{Proof of the main results}

We denote by $\cL(X)$ the space of arcs on $X$. Recall that an arc on $X$ is a morphism $\gamma:\Spec \Ct\to X$ of $\bC$-schemes.  We denote by 
$$
\pi_m:\cL(X)\to \cL_m(X)\quad\text{ and }\quad \pi_{m}^l:\cL_l(X)\to \cL_m(X)
$$
the truncation morphisms, for $0\le m\le l$. We let from now on $$\sX_m:=\sX_m(f,\Sigma),$$ and
$$
\sX_m^\infty  :=\pi_m^{-1}(\sX_m )\quad\text{ and }\quad  \sX_m^l :=(\pi_m^l)^{-1}(\sX_m ).
$$

For an arc $\gamma$ on $X$, let $\gamma(0):=\pi_0(\gamma)$ denote the center of $\gamma$, that is, the image of the closed point of $\Spec \Ct$ under $\gamma$, and by $f(\gamma)(t)$ we denote the power series associated to the composition $f\circ \gamma$.

If $\gamma\in \mathscr X_{m}^\infty $, then clearly $f_{red}(\gamma)(t)\ne 0$, that is, the generic point of the arc $\gamma$ lies in $X\setminus X_0$. Thus, by the valuative criterion of properness for the map $h:Y\ra X$, there exists a unique lifting of $\gamma$ to an arc $\tilde{\gamma}$ of $Y$. We define
$$\mathscr X_{m,i}^\infty :=\{\gamma\in \mathscr X_{m}^\infty  \mid \tilde{\gamma}(0)\in E_i^{\circ}\}.$$

\begin{lemma}\label{lemD}
There is a decomposition into mutually disjoint subsets
$$\mathscr X_{m}^\infty =\bigsqcup_{i\in S_m} \mathscr X_{m,i}^\infty .$$
Moreover, each $ \mathscr X_{m,i}^\infty $ with $i$ in $S_m$ is a constructible cylinder in $\cL(X)$, that is, the inverse image under $\pi_m$ of a constructible subset of $\cL_m(X)$.
\end{lemma}
\begin{proof} If $\gamma$ is in $\mathscr X_{m}^\infty $, $\tilde\gamma(0)$ must lie in some $E_i$ with $i\in A$. If $\tilde\gamma(0)$ lies also in $E_j$ for some $j\in S$ with $i\ne j$, then  $\tilde\gamma(0)\in E_i\cap E_j$ and so $E_i\cap E_j \ne 0$. Since the resolution is $m$-separating, $m_i+m_j>m$. Hence $\gamma$ has contact order $\ge m_i+m_j>m$ with $f$, which is a contradiction. Thus $\tilde\gamma(0)\in E_i^\circ$ and the disjoint decomposition follows.

Equivalently, the decomposition follows from \cite[Theorem A]{ELM}, by noting that 
$$\mathscr X_{m,i}^\infty =h_\infty (\text{Cont}^{\mu}(E))\cap \mathscr X_{m}^\infty $$
in the notation of {\it loc. cit.}, see Section \ref{secApp}  below, where
$$
h_\infty:\cL(Y)\ra \cL(X)
$$
is the induced map on arc spaces, $\mu=(\mu_j)_{j\in S}$ with $\mu_j=0$ if $j\ne i$ and $\mu_i=k_i$, 
and 
$$
\text{Cont}^{\mu}(E) =\{\lambda\in \cL(Y)\mid \text{ord}_\lambda E_j = \mu_j\text{ for all }j\in S\}.
$$
Moreover  $h_\infty (\text{Cont}^{\mu}(E))$ is a constructible cylinder in $\cL(X)$ by {\it loc. cit.}, and hence $\mathscr X_{m,i}^\infty $ is also a constructible cylinder. 
\end{proof} 

\begin{rem}\label{remEL} We can say more precisely that each $ \mathscr X_{m,i}^\infty $ with $i$ in $S_m$ is the pull-back of a constructible subset of $\cL_l(X)$ for any $l\ge \max\{2k_i(\nu_i-1), k_i(\nu_i-1)+m\}$. This follows from the proof of \cite[Corollary 1.7]{ELM}.
\end{rem}

Let now $l$ be a positive integer so that each $ \mathscr X_{m,i}^\infty $ with $i$ in $S_m$ is the pull-back of a constructible subset of $\cL_l(X)$. We let $$\mathscr X_{m,i}^l :=\pi_l(\mathscr X_{m,i}^\infty ).$$ Then $\mathscr X_{m,i}^l $ is a constructible subset of $\cL_l(X)$ such that
$$
\mathscr X_{m,i}^\infty  = \pi_l^{-1}(\mathscr X_{m,i}^l )
$$
and
$$\mathscr X_{m}^l =\bigsqcup_{i\in S_m} \mathscr X_{m,i}^l .$$
We construct a filtration $F_p\mathscr X_m^\infty $ of $\mathscr X_m^\infty $. For an integer $p$, we let
$$F_{p}\mathscr X_m^\infty := \bigsqcup_{i\in S_m,\, w_ik_i\geq -p} \mathscr X_{m,i}^\infty ,$$
$$\ F_{(p)}\mathscr X_m^\infty := F_{p}\mathscr X_{m}\setminus F_{p-1}\mathscr X_{m}=\bigsqcup_{i\in S_{m,p}}\mathscr X_{m,i}^\infty .$$
We define similarly a filtration  $F_p\mathscr X_m^l $ of $\mathscr X_m^l $ by replacing $\infty$ with $l$.

\begin{lemma}\label{lm21} Given $m> 0$, for all $l\gg 0$ we have: the set $F_p\mathscr X_m^l $ is Zariski closed in $\mathscr X_m^l $ for every integer $p$.
\end{lemma}
\begin{proof}
For $l\gg 0$ we have that $F_p\mathscr X_m^\infty $ is the pullback of $F_p\mathscr X_m^l$ by the truncation map  $\pi_l:\cL(X)\to \cL_l(X)$.
Since $X$ is smooth, it admits a finite cover by affine Zariski open subsets such that each of them is an etale open subset of $\mathbb{C}^d$. Since the assertion is local in $X$ for the Zariski topology, we can assume that $X$ itself is an \'etale open subset of $\mathbb{C}^d$. Then the projection $\pi_l$ is a trivial fibration. Hence, it is enough to prove that $F_{p}\mathscr X_m^\infty$ is closed in $\mathscr X_{m}^\infty $.

Since $X$, $Y$ are smooth and $h$ is proper birational, we have $h_*\cO_Y\simeq\cO_X$. Since $X$ is affine there is an isomorphism of  rings of global regular functions
$$
h^\#:\Gamma(X,\cO_X) \xrightarrow{\sim}\Gamma(Y,\cO_Y),\quad \phi\mapsto \phi\circ h.
$$
Since $-W=\sum_{i\in S}w_iE_i$ is an effective divisor, $\mathcal{O}_Y(W)$ is a sheaf of ideals of $\cO_Y$. One has thus an ideal of regular functions on $X$
$$
 \mathcal{I} :=(h^\#)^{-1}(\Gamma(Y,\cO_Y(W))).
$$ 

If $\gamma$ is an arc on $X$ not completely lying in $X_0$, then its lift $\tilde{\gamma}$ to an arc of $Y$ satisfies
$$\text{ord}_\gamma \mathcal{I} =\text{ord}_{\tilde{\gamma}}(-W),$$
since $\cO_Y(W)$ is generated by its global sections. Then the set $F_{p}\mathscr X_m^\infty $ can be expressed as 
$$F_{p}\mathscr X_m^\infty =\mathscr X_{m}^\infty\cap \{\gamma\in \cL(X)  \mid \ord_{\gamma} (\mathcal{I}) \geq -p\},$$
which is closed in $\mathscr X_{m}^\infty $.

\end{proof}

\begin{rem}\label{remBBB}
One can determine a precise lower-bound for $l$ from $m$, $m_i$, and $w_i$, similarly to Remark \ref{remEL}. From now on we fix $l$ as in Lemma \ref{lm21}. 
\end{rem}

\begin{lemma}\label{lm22}
For any integer $p$ and any $i\in S_{m,p}$, the set  $\mathscr X_{m,i}^l $ is Zariski closed in  $F_{(p)}\mathscr X_{m}^l $.
\end{lemma}
\begin{proof}
Suppose by contradiction that $\mathscr X_{m,i}^l $ is not closed in $F_{(p)}\mathscr X_{m}^l $ for some $i\in S_{m,p}$. Then there exist $j\in S_{m,p}$ different than $i$, and $\gamma\in \mathscr X_{m,j}^l $ such that $\gamma$ is in the closure $\ol{\mathscr X_{m,i}^l }$ of $\sX_{m,i}^l $ in $F_{(p)}\mathscr X_{m}^l $. Since $\mathscr X_{m,i}^l $ is a constructible set, the usual curve selection lemma holds, see \cite{Mi}. That is, there exists a complex analytic curve germ
$$\alpha\colon (\mathbb{C},0)\to (\ol{\mathscr X_{m,i}^l },\gamma),\quad s\mapsto \alpha(s)$$
such that $\alpha(0)=\gamma$, and $\alpha(s)\in \mathscr X_{m,i}^l $ for all $s\neq 0$ close to $0$.

Fix a local section of the truncation morphism $\pi_l:\cL(X)\ra\cL_l(X)$ in a neighborhood of $\gamma$. Via this section, we view now $\gamma$ as an arc on $X$, and thus $\alpha$ defines a complex analytic surface germ, i.e. a wedge,
$$\alpha\colon (\mathbb{C}^2,(0,0))\to (X,\gamma(0)),\quad (t,s)\mapsto \alpha(t,s)$$
such that
\begin{itemize}
\item[(a)] $\alpha_0(t):=\alpha(t,0)=\gamma$,
\item[(b)] $\alpha_s(t):=\alpha(t,s)$ is an arc lifting to $Y$ with center on $E^{\circ}_i$ for all $s\neq 0$.
\end{itemize}
Consider the following diagram
\begin{displaymath}
\xymatrix{
&Y\ar[d]^{h}\ar@{<--}[dl]_{\beta} \\
\mathbb{C}^2\ar[r]_{\alpha}&X}
\end{displaymath}
which defines the meromorphic map $\beta=h^{-1}\circ \alpha$. The meromorphic map $\beta$ cannot be holomorphic. If so, then $\beta(t,s)$ would equal $\tilde{\alpha}_s(t)$ for all $s,t$, where $\tilde{\alpha}_s$ is the unique lifting of $\alpha_s$ to $Y$. The latter is however not even continuous in $s$: the lifting $\tilde{\alpha}_0$ of $\alpha_0$ has center $\tilde{\alpha}_0(0)\in E^{\circ}_j$, and the lifting $\tilde{\alpha}_s$ of $\alpha_s$ has center $\tilde{\alpha}_s(0)\in E^{\circ}_i$ for all  $s\neq 0$. Hence the map $\beta$ has non-trivial locus of indeterminacy, which, by a theorem of Remmert \cite[p.333]{Re}, is a complex analytic subspace of codimension $\ge 2$ since $\bC^2$ is normal. By Hironaka, the locus of indeterminacy of $\beta$ can be resolved by a sequence of  blow ups:
\begin{displaymath}\label{diag1}
\xymatrix{
Z\ar[r]^{\bar{\beta}}\ar[d]_{\sigma}&Y\ar[d]^{h}\ar@{<--}[dl]_{\beta} \\
\mathbb{C}^2\ar[r]_{\alpha}&X.
}
\end{displaymath}
Here $Z$ can be defined as a log resolution $Z\ra \bC^2\times_XY$ of the locus where the natural holomorphic map $\bC^2\times_XY\ra \bC^2$ fails to be biholomorphic.

Let $F=\sigma^{-1}(0)=\cup_{j\in J} F_j$ be the exceptional divisor of $\sigma$. Let $L_{s}$ be the line $\{(t,s)\mid t\in \bC\}$ in $\mathbb{C}^2$. Let $$\sigma^{*}L_0=\sum_{j\in J} b_jF_j+\tilde{L}_0$$ be the total transform of $L_0$ under $\sigma$, where $\tilde{L}_0$ is the strict transform of $L_0$. For $s\ne 0$, $L_s$ does not meet the origin, and so $\sigma^{*}L_s=\tilde{L}_s$. The composition $$Z\to \mathbb{C}^2\overset{\mathrm{pr}_2}{\to}\mathbb{C}$$  gives by definition a rational equivalence between the cycles determined by its fibers. Hence the cycles $\sigma^{*}L_s$ and $\sigma^{*}L_0$ are rationally equivalent for any $s$. As a consequence we have an equality of intersection products
$$[\sigma^{*}L_s]\cdot \bar{\beta}^{*}W=[\sigma^{*}L_0]\cdot \bar{\beta}^{*}W.$$
Note that, $\bar{\beta}^{*}W$ is supported on the exceptional divisor $F$, since $W$ is supported on the exceptional divisor $E$. Hence $\bar{\beta}^{*}W$ is a compact cycle, and thus the intersection product is well-defined. We have the following equalities
$$[\sigma^{*}L_s]\cdot \bar{\beta}^{*}W=[\tilde{L}_s]\cdot \bar{\beta}^{*}W=\bar{\beta}_{*}\tilde{L}_s\cdot [W]=\mathrm{ord}_{\tilde{\alpha}_s}(W)=p,$$
and 
\begin{align*}
[\sigma^{*}L_0]\cdot \bar{\beta}^{*}W&=[\tilde{L}_0]\cdot \bar{\beta}^{*}W+\sum b_j[F_j]\cdot \bar{\beta}^{*}W\\
&=\bar{\beta}_{*}[\tilde{L}_0]\cdot W+(\sum b_j \bar{\beta}_{*}F_j)\cdot W\\
&=\mathrm{ord}_{\tilde{\gamma}}(W)+(\sum b_j \bar{\beta}_{*}F_j)\cdot W\\
&=p+(\sum b_j \bar{\beta}_{*}F_j)\cdot W,
\end{align*}
where $\tilde{\gamma}=\tilde{\alpha}_0$.
Hence we obtain that $$(\sum b_j \bar{\beta}_{*}F_j)\cdot W=0.$$
Since the $b_j$ are non-negative, by Kleiman ampleness criterion we obtain $b_j=0$ for all $j$ such that $\bar{\beta}(F_j)$ is not collapsed to a point. But this means that the map $\beta$ has no indeterminacy, which is a contradiction.
\end{proof}

For the following we will fix possibly-higher value for $l$ than up to now, one such that Lemma \ref{lemLoo1} applies.

\begin{lemma}\label{lm23}
For every $i\in S_m$, $\mathscr X^l_{m,i} $ is a smooth complex variety of dimension $d(l+1)-k_i\nu_i-1$, and it has the same homotopy type as $\tilde{E}_i^{\circ}$.
\end{lemma}
\begin{proof}
We consider the induced morphisms
$$
h_\infty:\cL(Y)\ra \cL(X)\quad\text{and}\quad h_l\colon \cL_l(Y)\to \cL_l(X)$$ and denote  
$$\mathscr Y^\infty_{m,i}:=h_\infty^{-1}\left( \mathscr X^\infty_{m,i} \right)
\quad\text{and}\quad
\mathscr Y^l_{m,i}:=h_l^{-1}\left( \mathscr X^l_{m,i} \right).$$ 
Note that 
$$\mathscr Y^\infty_{m,i}=\{{\gamma}\in \cL(Y)\mid (f\circ h)({\gamma})= t^{m} + \text{(higher order terms)} \in\Ct, \text{ and }\gamma(0)\in E_i^\circ\},$$

The morphism 
$$\pi_0:\mathscr Y^\infty_{m,i} \ra {E}_i^\circ,\quad \gamma\mapsto \gamma(0)$$
factorizes through the cyclic cover $\tilde{E}_i^\circ\ra E_i^\circ$ and  a morphism
$$
\tilde{\pi}_0:\mathscr Y^\infty_{m,i} \ra \tilde{E}_i^\circ
$$
which we define as follows. Let $U$ be any open neighborhood in $Y$ of a point in $E^{\circ}_i$ such that $f\circ h=u\cdot y_i^{m_i}$ in $U$, where $y_i$ is a local equation for $E_i$ and $u$ is an invertible regular function on $U$. Then the restriction of $\tilde{\pi}_0$ on the open subset $(\pi_0)^{-1}(U)\cap \mathscr Y^\infty_{m,i}$ is given by
$$\tilde{\pi}_0({\gamma}):=\left(\mathrm{ac}(y_i({\gamma})),{\gamma}(0)\right),$$
where $\mathrm{ac}(y_i({\gamma}))$ denotes the coefficient of the lowest order power of $t$ in the power series $y_i({\gamma})$. To check that the image of $\tilde{\pi}_0$ lies indeed in $\tilde{E}_i^\circ$, note that the power series $(f\circ h)(\gamma)$ is
$$
t^m + \text{(higher order terms)} = u(\gamma)\cdot y_i(\gamma)^{m_i}.
$$
Thus
$$
(\text{ac}(y_i(\gamma)))^{m_i} = (\text{ac}(u(\gamma)))^{-1} = u(\gamma(0))^{-1}.
$$

Define
$$
\tilde{\pi}^l_0:\mathscr Y^l_{m,i}\ra \tilde{E}_i^\circ
$$
similarly to $\tilde{\pi}_0$. Since $\mathscr X^\infty_{m,i} =\pi_l^{-1}(\mathscr X^l_{m,i} )$, one has $\mathscr Y^\infty_{m,i} =\pi_l^{-1}(\mathscr Y^l_{m,i} )$ as well, where we abuse the notation and denote also by $\pi_l$  the map $\cL(Y)\ra\cL_l(Y)$. Since $l$ is very big, cf. Remark \ref{remBBB}, it follows that $\tilde{\pi}_0$ factorizes through $\tilde{\pi}^l_0$.

We consider now the following diagram 
\begin{displaymath}
\xymatrix{
\mathscr Y^l_{m,i}\ar[r]^{h_l\;\;\;}\ar[d]_{\tilde{\pi}_0^l}&\mathscr X^l_{m,i}  \\
\tilde E^{\circ}_i&
}
\end{displaymath}
In this diagram, the morphism $h_l$ is a locally trivial fibration with fiber $\mathbb{C}^{(\nu_i-1)k_i}$, by Lemma \ref{lemLoo1} below. The lemma is then completed by the following observation: the morphism $\tilde{\pi}_0^l$ is  a locally trivial fibration with fiber $\mathbb{C}^{dl-k_i}$.  To prove this claim, fix an open neighborhood $U$ in $Y$ of a point $P_0$ in $E_i^\circ$ as above. Note that $\tilde{E}_i^\circ\cap (\bC\times U)$ is the restriction above $E_i^\circ$ of the \'etale cyclic cover
$$
\tilde{U}=\{(z,P)\in \bC\times U\mid z^{m_i}=u(P)^{-1}\} \xrightarrow{p_2} U.
$$
Let $(z_0,P_0)$ be a fixed point in $\tilde{E}_i^\circ\cap\tilde{U}$ above $P_0$, and let $\Omega$ be a small open neighborhood of $(z_0,P_0)$ in $\tilde{U}$. Note that the projection onto the first coordinate defines a regular invertible function on $\tilde{U}$, whose inverse  we denote by $\tilde{u}$. The function $\tilde{u}$ plays the role of $u^{1/m_i}$, the latter being not necessarily well-defined on $U$; that is, $\tilde{u}$ satisfies $\tilde{u}^{m_i}=u\circ p_2$.  Then $\tilde{y}_i:=\tilde{u}y_i$ is a local equation for $\tilde{E}_i^\circ$ in $\Omega$, and $f\circ h \circ p_2 = \tilde{y}_i^{m_i}$. Since $\tilde{y}_i$ is smooth, it forms part of an \'etale local system of coordinates on $\Omega$. Since  forming of jet schemes is compatible with \'etale morphisms by \cite[Lemma 4.2]{DL99}, it follows that $\tilde{\pi}^l_0$ is trivialized above $\tilde{E}_i^\circ\cap \Omega$ with fiber isomorphic to
$$\{ \gamma \in \cL_l(\Omega) \mid \gamma(0)=(z_0,P_0)\text{ and } \mathrm{ac}(y_i({\gamma})= z_0 \text{ and } \tilde{y}_i^{m_i}(\gamma)\equiv t^{m}\text{ mod } (t^{m+1}) \}\simeq$$ 
$$\simeq\{ \gamma \in \cL_l(\Omega) \mid \gamma(0)=(z_0,P_0)\text{ and } \tilde{y}_i(\gamma)\equiv t^{k_i}\text{ mod } (t^{k_i+1}) \}\simeq$$ 
$$\simeq \{ \gamma \in \cL_l(\bC^d) \mid \gamma(0)=0\text{ and } x_1(\gamma)\equiv t^{k_i}\text{ mod } (t^{k_i+1}) \}  \simeq\bC^{dl-k_i}.$$
\end{proof}

\begin{lemma}\label{lemLoo1} The morphism $h_l: \mathscr Y^l_{m,i}\ra \mathscr X^l_{m,i}$ is a Zariski locally-trivial fibration with fiber $\mathbb{C}^{(\nu_i-1)k_i}$ for $l\gg 0$.
\end{lemma}
\begin{proof}  As in the proof of Lemma \ref{lemD}, the map $\mathscr Y^l_{m,i}\ra \mathscr X^l_{m,i}$ is obtained by base-change from the map 
$$
\pi_l(\Cont^\mu(E))\ra h_l\pi_l(\Cont^\mu(E)),
$$
since $\mathscr X^l_{m,i}=h_l\pi_l(\Cont^\mu(E))\cap \mathscr X^l_m$ for $l\gg 0$.
Thus the claim follows immediately from Theorem \ref{lemLoo} in the Appendix.
\end{proof}

\begin{theorem}\label{main2} Let $f:X\to\bC$ be a non-invertible regular function on a smooth complex algebraic variety $X$ of dimension $d$, $\Sigma$ a non-empty closed subset of $f^{-1}(0)$, $m> 0$ an integer, and $h:Y\to X$ an $m$-separating log resolution of $(f,\Sigma)$. For all $l\gg 0$, there is a cohomological spectral sequence 
$$E_1^{p,q}=\bigoplus_{i\in S_{m,p}} H_{2(d(l+1)-k_i\nu_i-1)-(p+q)} (\tilde{E^{\circ}_i},\mathbb{Z})$$
converging to $H^{p+q}_c\left(\mathscr X_{m}^l(f,\Sigma),\mathbb{Z}\right)$.
\end{theorem}
\begin{proof}
By Lemma \ref{lm21}, there is a finite filtration of $\mathscr X^l_{m} =\mathscr X_{m}^l(f,\Sigma)$ by Zariski closed subsets 
$$\mathscr X^l_{m} =F_0\mathscr X^l_{m} \supset F_{-1}\mathscr X^l_{m} \supset\cdots\supset F_{p}\mathscr X^l_{m} \supset F_{p-1}\mathscr X^l_{m} \supset\cdots.$$
This induces a spectral sequence converging to $H^{*}_c(\mathscr X^l_{m} ,\mathbb{Z})$ such that $$E_1^{p,q}=H^{p+q}_c(F_{(p)}\mathscr X^l_{m} ,\mathbb{Z}).$$ 
It follows from Lemma \ref{lm22} that $F_{(p)}\mathscr X^l_{m} =\bigsqcup_{i\in S_{m,p}}\mathscr X^l_{m,i} $ is the irreducible decomposition of $F_{(p)}\mathscr X^l_{m} $. Hence we have $$H^{p+q}_c(F_{(p)}\mathscr X^l_{m} ,\mathbb{Z})\simeq\bigoplus_{i\in S_{m,p}} H^{p+q}_{c} (\mathscr X^l_{m,i} ,\mathbb{Z}).$$ 
By the smoothness of $\mathscr X^l_{m,i} $ proven in Lemma \ref{lm23} and Poincare duality, this is in turn isomorphic to
$$\bigoplus_{i\in S_{m,p}} H_{2(d(l+1)-k_i\nu_i-1)-(p+q)} (\mathscr X^l_{m,i} ,\mathbb{Z}).$$
 Since $\mathscr X^l_{m,i} $ and $\tilde{E}^{\circ}_i$ are homotopy equivalent by Lemma \ref{lm23}, we obtain
$$H^{p+q}_c(F_{(p)}\mathscr X^l_{m} ,\mathbb{Z})\simeq \bigoplus_{i\in S_{m,p}} H_{2(d(l+1)-k_i\nu_i-1)-(p+q)} (\tilde{E}^{\circ}_i,\mathbb{Z}).$$
\end{proof}

\begin{proof}[Proof of Theorem \ref{main}] 
Since $X$ is smooth it admits a cover $\{X_i\}_{i\in I}$ by Zariski open subsets such that each $X_i$ is an \'etale open subset of $\bC^d$. Then, for each $i\in I$ there is an isomorphism $\cL_l(X_i)\simeq \cL_m(X_i)\times\bC^{d(l-m)}$ such that the truncation morphism  $\pi^l_m:\cL_l(X_i)\ra\cL_m(X_i)$ is the first projection. We conclude that $\pi^l_m:\cL_l(X)\ra\cL_m(X)$ is a Zariski locally trivial fiber bundle, with fiber $\bC^{d(l-m)}$. Since $\sX^l_m =(\pi^l_m)^{-1}\sX_m $, we have that $\sX^l_m$ is a locally trivial fiber bundle over $\sX_m $ with fiber $\bC^{d(l-m)}$. The trivializing open subsets are  $\sX_m\cap \cL_m(X_i)$ for $i\in I$.

Let $c_l$ and $c_m$ be the constant maps sending $\sX^l_m$ and $\sX_m $ to a point respectively. The spectral sequence of the composition of the functors $R(c_l)_{!}=R(c_m)_{!}R(\pi^l_m)_!$ has $$E_2^{p,q}=H^p_c(\sX_m,R^q(\pi^l_m)_!\bZ_{\sX^l_{m}})$$
and converges to 
$$
H^{p+q}_c(\sX_m^l,\bZ).
$$

Observe that $H^q_c(\bC^{d(l-m)},\bZ)\simeq \bZ$ if $q=2d(l-m)$ and it is equal to $0$ otherwise. Then  $R^q(\pi^l_m)_!\bZ_{\sX^l_{m}}$ is $0$ if $q\ne d(l-m)$. By Poincar\'e-Verdier duality \cite[Theorem 3.3.10]{Di} over the ring of coefficients $\bZ$, the rank one local system $R^{2d(l-m)}(\pi^l_m)_!\bZ_{\sX^l_{m}}$ is the dual of the local system  with fibers $H_0(\bC^{d(l-m)},\bZ)= \bZ\cdot [\textup{pt}]$ given by $\pi^l_m$ on $\sX_m$, where $[\textup{pt}]$ is the $0$-homology class of a point in a fiber of $\pi^l_m$. Since the affine space is path-connected, every loop in $\sX_m$ will send $[\textup{pt}]$ to itself, and hence $R^{2d(l-m)}(\pi^l_m)_!\bZ_{\sX^l_{m}}$ is the constant local system $\bZ$ on $\sX_m$. Therefore we obtain
$$
H^*_c(\sX^l_m ,\bZ) \simeq H^{*-2d(l-m)}_c(\sX_m ,\bZ).
$$
Thus, replacing $q$ with $q+2d(l-m)$ in Theorem \ref{main2}, one obtains the claim.
\end{proof}

\begin{proof}[Proof of Proposition~\ref{degen}]
The spectral sequence of a filtration by algebraic closed subsets of a complex algebraic variety can be lifted to the category of mixed Hodge structures by \cite[Lemma 3.8]{A}. Then the spectral sequence from Theorem \ref{main2}, after tensoring with $\bQ$, is a spectral sequence of $\bQ$-mixed Hodge structures once the Tate twists are taken into account. The rational structure is necessary in order to lift the Poincar\'e duality, cf. \cite{Ak}. More precisely, the proof of Lemma~\ref{lm23} shows that that there is an isomorphism of $\bQ$-mixed Hodge structures
$$H^k_c(\mathscr \sX^l_{m,i} ,\bQ)\simeq H^{2d_i-k}(\tilde{E}_i^\circ,\bQ)^\vee\otimes\bQ(-d_i)$$
for every integer $k$, where $d_i=\dim \sX^l_{m,i} $. Since $\tilde{E}^{\circ}_i$ is a smooth variety of dimension $d-1$, it follows that for non-zero groups $H^{k}_c(F_{(p)}\mathscr \sX^l_{m} ,\mathbb{Q})$, the weights on are contained in 
$$[k-d+1,k] \cap [0,k].$$ 
Since this interval has length at most $d$ and the differentials of the spectral sequence are morphisms of mixed Hodge structures, the result follows.
\end{proof}

\begin{proof}[Proof of Proposition~\ref{propMult}]
We prove that for every $q$, there is at most one non-zero term among $E^{p,q}_1$, and respectively among $'\!E^{p,q}_1$. 

Let $h_0: Y_0\ra X=\bC^d$ be the blowup at the origin. Let $E_0$ be the exceptional divisor. Let $D=f^{-1}(0)$, and let $\tilde{D}$ be the strict transform of $D$. Then
$$
\tilde{D}\cap E_0 \simeq f_m^{-1}(0)\subset \bP^{d-1}\simeq E_0.
$$
Moreover, since $f_m$ is a homogeneous polynomial, the closed subset $F=\{f_m=1\}$ of $\bC^d$ is the Milnor fiber of $f_m$. The associated map $$F\ra \bP^{d-1}\setminus f_m^{-1}(0)$$ to the complement of the zero locus in the projective space of $f_m$, sending a point to the line connecting it to the origin, displays $F$ as the cyclic cover $\tilde{E}_0^\circ$ of degree $m$ of $E_0^\circ=E_0\setminus \tilde{D}\cap E_0$.

Let $h:Y\ra X$ be an $m$-separating log resolution of $(f,\Sigma)$ factoring through $Y_0$, with $\Sigma$ consisting of the origin only. Abusing slightly the notation, we denote by $E_0$ the strict transform of $E_0$. Then the multiplicity of every exceptional divisor $E_i$ in $E=h^{-1}(0)$ is $>m$, and it is equal to $m$ only if $i=0$. Therefore the index set
$$
S_{m,p}=\left\{
\begin{array}{cc}
\{0\} & \text{ if } p=-w_0,\\
\emptyset & \text{ if } p\ne -w_0. 
\end{array}
\right.
$$
One obtains thus for every $q$ that
$$
E_1^{-w_0,q}= H_{2(dm-1)-(-w_0+q)}(\tilde{E}_0^\circ,\bZ) \simeq H_c^{-w_0+q}(\mathscr X_m(f,\Sigma),\bZ)
$$
by Theorem \ref{main},
and
$$
'\!E^{-w_0,q}_1= H_{-1-d-(-w_0+q)}(\tilde{E}_0^\circ,\bZ) \simeq  HF^{-w_0+q}(\phi^m,+)
$$
by (\ref{eqMcL}). The claim now follows.
\end{proof}


Finally, we provide a proof of \cite[Lemma 2.4]{Mclean} since the proof in {\it loc. cit.} contains a false claim (i.e. the inequality $a_Y-b_Y<a_{Y'}-b_{Y'}$ does not necessarily hold).

\begin{lemma}\label{lemMsep}
Let $X$ be a smooth complex algebraic variety of dimension $d$, $f$ a non-invertible regular function on $X$, $\Sigma$ a non-empty closed subset of $f^{-1}(0)$, and $m>0$ an integer. There exist $m$-separating log resolutions of $(f,\Sigma)$.
\end{lemma}
\begin{proof}
Let $h:Y\ra X$ be a log resolution of $(f,\Sigma)$ and let $E=(f\circ h)^{-1}(0)=\sum_{i\in S}m_iE_i$ be the pull-back of the divisor $f^{-1}(0)$ with $E_i$ the irreducible components of $E$. For $I\subset S$, let $E_I=\cap_{i\in I}E_i$.

Let $\Delta$ be the dual complex of the simple normal crossings divisor $E$, whose vertices are labelled by $S$ and for every irreducible component of an $E_I$ with $I\subset S$ one attaches a $(|I|-1)$-dimensional cell. Each vertex $i\in S$ of $\Delta$ comes with the associated multiplicity $m_i$. For a cell $\sigma$ of $\Delta$, the multiplicity is defined $m_\sigma$ as the sum of the multiplicities of its vertices. Define
$$M(\Delta)=\min\{m_\sigma \mid \sigma\text{ is a 1-cell of }\Delta\}$$
and 
$$
\cM(\Delta)=\{ \text{ 1-cells } \sigma\text{ of }\Delta\mid m_\sigma=M(\Delta)\}.
$$

Consider the blowup $Y'$ of $Y$ along a non-empty intersection $E_i\cap E_j$ corresponding to a 1-cell $\sigma$ in $\cM(\Delta)$. Let $\Delta'$ be the associated dual complex. Then $\Delta'$ is the stellar subdivision of $\Delta$ at $\sigma$, see \cite[\S 9]{K}. The multiplicity of the new vertex in $\Delta'$ is $m_\sigma=m_i+m_j$. The multiplicities of the other vertices remain the same. Therefore $M(\Delta')\ge M(\Delta)$. 

If $M(\Delta')=M(\Delta)$, then $\cM(\Delta')= \cM(\Delta)\setminus\{\sigma\}$. We repeat the blowing up process $|\cM(\Delta)|-1$ times. The resulting modification $Y'\ra X$ then satisfies that $M(\Delta')>M(\Delta)$. 

Repeating this process, we have achieve after finitely many steps that $M(\Delta')>m$. \end{proof}

\section{Appendix}\label{secApp}

The goal of this appendix is to prove Theorem \ref{lemLoo}, a particular case of the local triviality claimed in \cite[Lemma 9.2]{Loo} whose proof was deemed ``easy" and not included.

The notation in this appendix is independent of the notation from the introduction.

For a smooth complex variety $X$, an ideal subsheaf $I$ of $\cO_X$ corresponding to a closed subscheme $Z$, and a natural number $m$, we denote the (non-restricted) $m$-contact locus by
$$
\Cont^m(Z)=\Cont^m(I)=\{\gamma\in\cL(X)\mid \ord_\gamma(I)=m\}.
$$
For a reduced closed subscheme with the decomposition $D=\cup_i D_i$ into irreducible components, and for a tuple $\bm=(m_i)_i$ of natural numbers, the (non-restricted) multi-contact loci are defined as
$$
\Cont^\bm(D)=\{\gamma\in\cL(X)\mid \ord_\gamma(D_i)=m_i\text{ for all } i\}.
$$

We recall that by $\pi_l:\cL(X)\ra\cL_l(X)$ we mean the truncation map from arcs to $l$-jets on $X$, and for a morphism $\mu:Y\ra X$, the associated maps on arcs and jets are denoted by $\mu_\infty:\cL(Y)\ra\cL(X)$ and $\mu_l:\cL_l(Y)\ra\cL_l(X)$, respectively.

The contact loci are cylinders since $\Cont^m(Z)=\pi_l^{-1}\pi_l (\Cont^m(Z))$ for $l\ge m$, and $\Cont^\bm(D)=\pi_l^{-1}\pi_l (\Cont^\bm(D))$ for $l\ge\max\{m_i\}$.

\begin{lemma}\label{lemmaBl1}
Let $2\le \nu\le d$ be integers and let $$\mu:Y=\bA^d\ra X=\bA^d$$ be the map given in coordinates by $x_i=y_dy_i$ for $i=1,\ldots, \nu-1$, and $x_i=y_i$ for $i=\nu,\ldots, d$.  Let $m\in\bN$. Then the map
$$\mu_l:\pi_l(\Cont^m(y_d))\ra \mu_l\pi_l(\Cont^m(y_d))$$
is a trivial fibration with fiber $\bC^{(\nu-1)m}$ for $l\ge m$.
\end{lemma}
\begin{proof} Note that $\mu:Y\ra X$ is an affine chart of the blowing up of $X$ at a linear subspace of codimension $\nu$, and $y_d$ is a defining function for the exceptional divisor.  Consider the exact sequence
\begin{equation}\label{eqHO2}
\mu^*\Omega_X\xrightarrow{d\mu} \Omega_Y \ra \Omega_{Y/X}\ra 0.
\end{equation}
Then
$$
d\mu=\left(
\begin{array}{cc}
y_d I_{\nu-1} & A\\
O & I_{d-\nu+1}
\end{array}
\right)
$$
where $I_k$ is the $k\times k$ identity matrix, $O$ is a zero matrix, and $A$ is the matrix with only the last row nonzero
$$
A=\left(
O
\begin{array}{c}
y_1\\
\vdots\\
y_{\nu-1}
\end{array}
\right).
$$
Hence, multiplying $d\mu$ on the left by the appropriate invertible matrix, we can effect a change of basis so that
$$
d\mu=\left(
\begin{array}{cc}
y_d I_{\nu-1} & O\\
O & I_{d-\nu+1}
\end{array} \right).$$
For any arc $\gamma\in \Cont^m(y_d)$, one therefore has
$$
\gamma^*(dh)=
\left(
\begin{array}{cc}
u_\gamma t^m I_{\nu-1} & O\\
O & I_{d-\nu+1}
\end{array} \right),
$$
for some invertible $u_\gamma\in \Ct$. 

It is shown in the proof of \cite[Lemma 9.2]{Loo} that the fiber of $$\mu_l:\pi_l(\Cont^m(y_d))\ra \mu_l\pi_l(\Cont^m(y_d))$$ through $\pi_l(\gamma)$ is the vector space
$$
K_l(\gamma):=\Hom_{\Ct}(\gamma^*\Omega_{Y/X}, \Ct/(t^{l+1})).
$$
This is the kernel of the linear map
$$
\Hom_{\Ct}(\gamma^*\Omega_{Y}, \Ct/(t^{l+1})) \ra \Hom_{\Ct}((h\circ\gamma)^*\Omega_{X}, \Ct/(t^{l+1}))
$$
induced $\gamma^*(d\mu)$ via the exact sequence (\ref{eqHO2}).  Hence
$$
K_l(\gamma) =[t^{l-m+1}\Ct/(t^{l+1})]^{\times \nu-1}\times \{0\}^{^{\times d-\nu+1}}
$$
does not depend on the choice of $\gamma$, and it is isomorphic to $\bC^{(\nu-1)m}$ as a $\bC$-vector space.
\end{proof}

\begin{proof}[The second proof]
It suffices to prove the case $d=\nu=2$, because the general case for $d$ and $\nu$ is similar. Let us consider the morphism $\mu: \mathbb A^2\to \mathbb A^2$ given by $(x,y)\mapsto(xy,y)$. By definition, $\Cont^m(y)$ is the subscheme of $\mathscr L(\mathbb A^2)$ consisting of $(\varphi,\psi)$ with $\psi$ of order $m$. Thus we can identify $\pi_l(\Cont^m(y))$ with
$$\{(a_0,\dots,a_l; b_m,\dots,b_l)\in \mathbb A^{2l-m+2} \mid b_m\not=0\},$$
and identify $\mu_l: \pi_l(\Cont^m(y))\ra \mu_l\pi_l(\Cont^m(y_d))$ with 
$$(a_0,\dots,a_l; b_m,\dots,b_l)\mapsto (a_0b_m,a_0b_{m+1}+a_1b_m,\dots,a_0b_l+\cdots+a_{l-m}b_m; b_m, \dots,b_l).$$
We now consider the morphism 	
$$(\mu_l,\pr_{\mathbb A^m}):\pi_l(\Cont^m(y))\ra \mu_l\pi_l(\Cont^m(y_d))\times_k \mathbb A^m$$
sending $(a_0,\dots,a_l; b_m,\dots,b_l)$ to
$$(a_0b_m,a_0b_{m+1}+a_1b_m,\dots,a_0b_l+\cdots+a_{l-m}b_m; b_m, \dots,b_l; a_{l-m+1},\dots, a_l).$$
This morphism is an isomorphism because given $(u_m,\dots,u_l; b_m,\dots,b_l; a_{l-m+1},\dots, a_l)$ in the target, the following system of linear equations, in variables $(a_0,\dots,a_{l-m})$,
\begin{equation*}
\begin{cases}
a_0b_m &=\ u_m\\
a_0b_{m+1}+a_1b_m &=\ u_{m+2}\\
\cdots \cdots &\quad\ \cdots\\
a_0b_l+\cdots+a_{l-m}b_m&=\ u_l
\end{cases}
\end{equation*}
is a Cramer system due to $b_m\not=0$.
\end{proof}

\begin{lemma}\label{lemBC1}
Let $\mu:Y\ra X=\bA^d$ be the blowing up of $X$ along a linear subspace $Z=\bA^{d-\nu}$. Let $E$ be the exceptional divisor. Let $m\in\bN$. Then the map
$$\mu_l:\pi_l(\Cont^m(E))\ra \mu_l\pi_l(\Cont^m(E))$$
is a Zariski locally-trivial fibration with fiber $\bC^{(\nu-1)m}$ for $l\ge m$.
\end{lemma}
\begin{proof}
In this case, $\mu_\infty(\Cont^m(E))=\Cont^m(Z)$ since $\mu$ is a log resolution of $(X,Z)$, by \cite[Theorem A]{ELM}. 

There is an open covering $Y=\cup_{i=1}^\nu Y_i$ with $Y_i\simeq \bA^d$ such that the restriction $\mu:Y_i\ra X$ is as in Lemma \ref{lemmaBl1} for a suitable choice of coordinates. For each $i$, the set
$$
\pi_l(\Cont^m(E\cap Y_i))=\{\gamma\in \pi_l(\Cont^m(E))\mid \gamma(0)\in Y_i\}
$$
is open in $\pi_l(\Cont^m(E))$.  

We show now that the set
$
\mu_l\pi_l(\Cont^m(E\cap Y_i))
$
is also open in $\mu_l\pi_l(\Cont^m(E))$.  It consists of truncations $\pi_l(\gamma)$ of arcs $\gamma\in\Cont^m(Z)$ with the tangent line at the center $\gamma(0)$  contained in the normal directions to $Z$ given by $E\cap Y_i$.

Letting $Z$ be the zero locus of $I=(x_1,\ldots,x_\nu)$, where $x_1,\ldots ,x_d$ are coordinate functions on $X$, the last condition means the following. The exceptional divisor $E\simeq\bP^\nu\times Z$ is the projectivization of the normal bundle of $Z$ in $X$, which is trivial. An arc $\gamma$ centered on $Z$ gives a well-defined point of $E$, the center of the unique lifting to an arc on $Y$,
$$
\tilde{\gamma}(0)=\left.\frac{d\gamma}{dt}\right\rvert_{t=0} \times \gamma(0)\quad\in\quad \bP^\nu\times Z=E.
$$
where
$$
\left.\frac{d\gamma}{dt}\right\rvert_{t=0} = \left.\left[\frac{d (x_1(\gamma))}{dt},\ldots ,\frac{d (x_\nu(\gamma))}{dt}\right]\right\rvert_{t=0} \quad\in\quad\bP^\nu
$$
is the point with homogeneous coordinates given by the coefficients of $t^{v-1}$ of the derivatives of $x_i(\gamma)(t)$ with respect to $t$, with $$v=\min\{\ord_\gamma x_i\mid i=1,\ldots,\nu\}=\ord_\gamma(I).$$
Thus, for a fixed $i$ in $\{1,\ldots, \nu\}$, the point $\tilde{\gamma}(0)$ belongs to $E\setminus (E\cap Y_i)$ if and only if the respective homogeneous coordinate of $(d\gamma/dt)\rvert_{t=0}$ is zero. That is, if and only if $\ord_\gamma x_i>v$. The last condition is a Zariski-closed condition on arcs, as well as on $l$-jets for $l\ge v$.

Thus for a fixed $i$ in $\{1,\ldots, \nu\}$,
$$
\mu_l\pi_l(\Cont^m(E\cap Y_i))=\mu_l\pi_l(\Cont^m(E))\cap \{\gamma\in\cL_l(X)\mid \gamma(0)\in Z \text{ and }\ord_\gamma x_i\le m\}
$$
is open in $\mu_l\pi_l(\Cont^m(E))$ for $l\ge m$.

Since
$$
\mu_l:\pi_l(\Cont^m(E\cap Y_i))\ra\mu_l\pi_l(\Cont^m(E\cap Y_i))
$$
is a trivial fibration with fiber $\bC^{(\nu-1)m}$ by Lemma \ref{lemmaBl1}, the claim follows.
\end{proof}

\begin{lemma}\label{lemBlow}
Let $\mu:Y\ra X$ be the blowing up of a smooth variety $X$ along a smooth subvariety $Z$. Let $E$ be the exceptional divisor. Let $m\in\bN$. Then the map
$$\mu_l:\pi_l(\Cont^m(E))\ra \mu_l\pi_l(\Cont^m(E))$$
is a Zariski locally-trivial fibration with fiber $\bC^{(\nu-1)m}$ for $l\ge m$.
\end{lemma}
\begin{proof}
It is enough to cover $X$ by open affine subsets $U$ and prove the claim for the restriction $\mu^{-1}(U)\ra U$ instead of $Y\ra X$. Let $\nu$ be the codimension of $Z$. Since the sequence of locally free sheaves
$$
0\ra I_Z/I_Z^2\ra \Omega_{X} \ra \Omega_Z \ra 0
$$ 
is exact, we can find $U$ and sections $x_1,\ldots, x_d$ of $\cO_X(U)$ with $x_1,\ldots, x_\nu$ generating $I_Z(U)$,   $dx_1,\ldots, dx_d$ trivializing $\Omega_X(U)$, and $dx_1,\ldots, dx_\nu$ trivializing $\Omega_Z(U)$. Thus we obtain a base change diagram
$$
\xymatrix{
Z\cap U \ar[r]\ar[d] & \bA^{d-\nu} \ar[d]\\
U \ar[r] & \bA^{d}.
}
$$
with the horizontal morphism being \'etale and the vertical ones closed immersions. Thus $\mu^{-1}(U)\ra U$ is the \'etale base change of the blowing up of $\bA^d$ along the linear subspace $\bA^{d-\nu}$. The claim follows from Lemma \ref{lemBC1} and the compatibility of jet schemes with \'etale morphisms.
\end{proof}

In what follows, the definition of a log resolution is more relaxed than in the introduction: we do not assume it anymore to be an isomorphism outside a fixed closed locus.

\begin{theorem}\label{lemLoo}
Let $\mu:Y\ra X$ be a log resolution of an ideal subsheaf $I$ of $\cO_X$ obtained by successively blowing up smooth centers. Let $I\cdot\cO_Y=\cO_Y(-\sum_i m_iE_i)$, where $E_i$ are the irreducible components of the zero locus $E$ of $I\cdot\cO_Y$, and $m_i\in \bN$. Let $\bk=(k_i)_i$ be a tuple of natural numbers. Then  for $l\gg 0$ the map
$$
\mu_l:\pi_l(\Cont^\bk(E))\ra\mu_l\pi_l(\Cont^\bk(E))
$$
if a Zariski locally-trivial fibration with fiber $\bC^e$ with $e=\sum_ik_i\cdot\ord_{K_{Y/X}}E_i$.
\end{theorem}
\begin{proof}
The proof of \cite[Lemma 9.2]{Loo} covers the claim except the fibration is only showed to be piecewise locally-trivial.

Factor $\mu:Y\ra X$ into 
$$
Y=Y^N\xra{\mu^N} Y^{N-1}\ra\ldots\xra{\mu^2} Y^1\xra{\mu^1} Y^0=X,
$$
with $$\mu^j:Y^j\ra Y^{j-1}$$ the blowing up along a smooth closed subvariety $Z^{j-1}$ of $Y^{j-1}$. Let $E^j$ be the exceptional divisor introduced by $\mu^j$.

Let $$\mu^{j,k}=\mu^k\circ\mu^{k+1}\circ\ldots\circ\mu^j:Y^j\ra Y^k $$ for $k\le j$. 
Define $$\sY^j=\mu^{N,j}_\infty (\Cont^\bk(E))\quad\subset \cL(Y^j),$$
so that we obtain a tower of surjective maps
$$
\Cont^\bk(E)=\sY^N \ra  \sY^{N-1}\ra\ldots\ra\sY^1\ra\sY^0=\mu_\infty(\Cont^\bk(E))
$$
induced by the maps $\mu^j_\infty:\cL(Y^j)\ra\cL_\infty(Y^{j-1})$.

By \cite[Theorem 2.1]{ELM} applied to the proper birational map $\mu^{N,j}:Y=Y^N\ra Y^j$, each $\sY^j$ is a cylinder. The proof of this fact also shows that $\Cont^\bk(E)=\sY^N$ is a union of fibers of $\mu^{N,j}_\infty:\cL_\infty(Y^N)\ra\cL_\infty(Y^j)$ for each $j$. That is,
$\sY^N=(\mu^{N,j}_\infty)^{-1}(\sY^j)$ for all $j$. Thus $$\sY^j=(\mu^j_\infty)^{-1}(\sY^{j-1})$$ for all $j$ as well. Since $\mu^{N,j}$ is a log resolution of $(Y^j,Z^{j})$ and of $(Y^j,E^j)$, by {\it loc. cit.} we have that
$$
\sY^j\subset \Cont^{p_j}(Z^{j})\quad\text{ and }\quad \sY^j\subset \Cont^{q_j}(E^j),
$$
for
$$
p_j=\sum_i k_i\cdot \ord_{Z^{j}}E_i\quad\text{ and }\quad q_j=\sum_ik_i\cdot\ord_{E^j}E_i = p_{j-1}.
$$

The cylinder property allows us to draw the same conclusions for $l$-jets for $l\gg 0$. Define $$\sY^j_l:=\pi_l(\sY^j).$$  Then for $l\gg 0$,
$$
\sY^j=\pi_l^{-1}(\sY^j_l)
$$
and we have a tower of surjective maps
$$
\pi_l(\Cont^\bk(E))=\sY^N_l \ra  \sY^{N-1}_l\ra\ldots\ra\sY^1_l\ra\sY^0_l=\mu_l\pi_l(\Cont^\bk(E)).
$$
induced by the maps $\mu^j_l:\cL_l(Y^j)\ra\cL_l(Y^{j-1})$, such that
$$
\sY^j_l=(\mu^j_l)^{-1}(\sY^{j-1}_l),
$$
$$
\sY^j_l\subset \pi_l(\Cont^{p_j}(Z^{j}))\quad\text{ and }\quad \sY^j_l\subset \pi_l(\Cont^{q_j}(E^j)).
$$
Therefore the map $\sY^j_l\ra\sY^{j-1}_l$ is obtained by base change from the map
$$
\pi_l(\Cont^{q_j}(E^j))\ra \mu^j_l\pi_l(\Cont^{q_j}(E^j))=\pi_l(\Cont^{p_{j-1}}(Z^{j-1})),
$$
which is Zariski locally-trivial by Lemma \ref{lemBlow}. Thus each map $\sY^j_l\ra\sY^{j-1}_l$ is a Zariski locally-trivial fibration, and so the composition $\sY^N_l\ra\sY^0_l$ is as well.
\end{proof}

\noindent
{\bf Acknowledgement.} We thank J. Sebag and the referees for useful comments.

N.B. was partly supported by the grants STRT/13/005 and Methusalem METH/15/026 from KU Leuven, G097819N and G0F4216N from FWO.

J.F.B. was supported by ERCEA 615655 NMST Consolidator Grant, MINECO by the project 
reference MTM2016-76868-C2-1-P (UCM), by the Basque Government through the BERC 2018-2021 program and Gobierno Vasco Grant IT1094-16, by the Spanish Ministry of Science, Innovation and Universities: BCAM Severo Ochoa accreditation SEV-2017-0718 and by Bolsa Pesquisador Visitante Especial (PVE) - Ciencias sem Fronteiras/CNPq Project number:  401947/2013-0.

L.Q.T. was supported by the grants mentioned for J.F.B. 

H.D.N. was supported by the grants mentioned for J.F.B, by Juan de la Cierva Incorporaci\'on IJCI-2016-29891, and the National Foundation for Science and Technology Development (NAFOSTED), Grant number 101.04-2019.316, Vietnam.

\end{document}